\documentclass[reqno]{amsart}

\usepackage{amsmath,amssymb,mathbbol}

\usepackage[active]{srcltx}
\usepackage{graphicx,color}

\usepackage{epic}
\usepackage{pstricks}
\usepackage{amsthm}

\newtheorem{theorem}{Theorem}[section]
\newtheorem{proposition}[theorem]{Proposition}

\newtheorem{lemma}[theorem]{Lemma}

\newtheorem{corollary}[theorem]{Corollary}
\theoremstyle{definition}
\newtheorem{definition}[theorem]{Definition}

\newtheorem{remark}[theorem]{Remark}

\newtheorem{question}{Question}

\def\R{\mathbb{R}}

\def\N{\mathbb{N}}
\def\C{\mathbb{C}}

\def\E{\mathbb{E}}

\def\lam{\lambda}

\def \supp{\operatorname{supp}}

\def\Re{{\mathrm{Re\,}}}

\def\cE{{\mathcal E}}

\def\cL{{\mathcal L}}
\def\cM{{\mathcal M}}

\newcommand{\ud}{\mathrm{d}}

\newcommand{\X}{{X}}

\newcommand{\B}{{B}}

\newcommand{\dom}[1]{\mathsf{D}_{#1}}  
\newcommand{\rg}[1]{\mathsf{R}_{#1}}

\def\trXY1-tq{Tr(X,Y, 1-\theta ,q)}

\def\mrew{W^{1,\Phi} (X,Y)}

\definecolor{darkred}{rgb}{0.7,0.1,0.1}

\title[Real interpolation with weighted rearrangement invariant BFS]{Real interpolation with weighted rearrangement invariant Banach function spaces}

\author{Ralph Chill}

\address{R.~Chill, Institut f\"ur Analysis, Fachrichtung Mathematik, TU Dresden, 01062 Dresden, Germany}
\email{ralph.chill@tu-dresden.de}

\author{Sebastian Kr\'ol}
\address{S. Kr\'ol, Institut f\"ur Analysis, Fachrichtung Mathematik, TU Dresden, 01062 Dresden, Germany and Faculty of Mathematics and Computer Science, Nicolaus Copernicus University, ul. Chopina 12/18, 87-100 Toru\'n, Poland}
\email{sebastian.krol@mat.umk.pl}

\thanks{The second author was partially supported by the Alexander von Humboldt Foundation and Narodowe Centrum Nauki grant DEC-2011/03/B/ST1/00407}

\begin{document}

\date{\today}



\begin{abstract} 
Motivated by recent applications of weighted norm inequalities to maximal regularity of first and second order Cauchy problems, we study real interpolation spaces on the basis of general Banach function spaces and, in particular, weighted rearrangement invariant Banach function spaces. We show equivalence of the trace method and the $K$-method, identify real interpolation spaces between a Banach space and the domain of a sectorial operator, and reprove an extension of Dore's theorem on the boundedness of $H^\infty$-functional calculus to this general setting.
\end{abstract}

\renewcommand{\subjclassname}{\textup{2010} Mathematics Subject Classification}

\maketitle

\section{Introduction}

In the interpolation theory the real interpolation spaces play a prominent role because they appear naturally as trace spaces of, for instance, classical Sobolev spaces on domains and are thus connected to the regularity theory of elliptic or parabolic boundary value problems. In the abstract setting, they appear as trace spaces of certain weighted, vector-valued Sobolev spaces on the real half-line and find here their connection to abstract Cauchy problems in Banach spaces, namely as the spaces of initial values for which the solutions have some described time regularity; see, for example, the monographs by Butzer \& Berens \cite{BuBe67}, Lunardi \cite{Lu95,Lu09} and Triebel \cite{Tr95}.

It is this connection to abstract Cauchy problems in combination with some recent results from Chill \& Fiorenza \cite{ChFi14} and the present authors \cite{ChKr15,ChKr14} on the extrapolation of $L^p$-maximal regularity to maximal regularity in weighted rearrangement invariant Banach function spaces which is the motivation of this article. We consider real interpolation spaces defined by the trace method, however, with the usual power weighted $L^p$ spaces on the half-line replaced by general Banach function spaces over $\R_+$. Real interpolation spaces in this general setting have already been studied, although they are usually defined by Peetre's $K$-method. Bennett, for example, extended the $K$-method by replacing the classical weighted $L^p$ spaces by {\em unweighted} rearrangement invariant Banach function spaces \cite{Be74II}. On the other hand, several authors studied extensions of the $K$-method using more general weights than power weights, staying however within the $L^p$ 
scale; see for example Kalugina \cite{Ka75} (quasipower weights), Sagher \cite{Sa81} (weights of Calder\'on type), and more recently Bastero, Milman \& Ruiz \cite{BaMiRu01} (Calder\'on weights). In \cite{BaMiRu01} were also considered the classes of weights $w$, introduced in Ari\~{n}o \& Muckenhoupt \cite{ArMu90}, such that the Hardy operator is bounded on the cone of decreasing functions in $L^p_w$. Real interpolation spaces on the basis of general Banach function spaces have been studied in Bennett \& Sharpley \cite{BeSh88} and Brudny{\u\i} \& Krugljak \cite{BrKr91}.

Given an interpolation couple $(X,Y)$ of Banach spaces, and given a Banach function space $\Phi$ over $\R_+$, we recall in Section \ref{sec.definition} the definitions of the real interpolation spaces $[X,Y]_{\Phi}$ and $(X,Y)_{\Phi}$ by the trace method and the $K$-method, respectively, and we show that they coincide under some assumptions on $\Phi$ which include the boundedness of the Hardy operator. Although the content of this section is straightforward, we include it, mainly because our setting slightly differs from the setting in Brudny{\u\i} \& Krugljak, but also because the Hardy operator is a main tool in the rest of this article. 

Like in the classical case, the identification of the real interpolation spaces by other methods is an important issue. In the literature there exists a number of results on the identification of the classical real interpolation spaces between the domains of powers of sectorial operators and the underlying Banach space. Sectorial operators arise in the theory of partial differential equations as well as in the abstract setting, for example as negative generators of $C_0$-semigroups. A representation of the classical interpolation spaces as \emph{Besov type spaces} was studied by Komatsu \cite{Ko67} (see also Triebel \cite{Tr95}), Haase \cite[Chapter 6]{Hs06}, and Kalton \& Kucherenko \cite{KaKu07}; see also relevant results due to Kriegler \& Weis \cite{KrWe14} and Kunstmann \& Ullmann \cite{KuUl14}. In Section \ref{sec.functional.calculus}, given a sectorial operator $A$ on a Banach space $X$, with domain $\dom{A}$, we characterise the real interpolation spaces $(X,\dom{A})_{\Phi}$ via the holomorphic 
functional calculus (Theorems \ref{fcrepresent} and \ref{fcrepresent2}) and via regularity of semigroup orbits (Theorem \ref{semigroup}). We then apply these 
characterisations in order to recover a consequence of a version of a theorem of Dore, due to Kalton \& Kucherenko \cite{KaKu07}, by showing that the part of a sectorial operator $A$ in the interpolation space $(X,\dom{A})_{\Phi}$ always admits a bounded $H^\infty$-functional calculus (Theorem \ref{general Dore th}).

As mentioned above, the boundedness of the Hardy operator (and sometimes also its formal adjoint and the Calder\'on operator) on the underlying Banach function space is a crucial ingredient in various proofs. This follows already from a close analysis of Bennett's approach. Section \ref{sec.weights} is thus devoted to the study of the boundedness of the Hardy operator on weighted rearrangement invariant Banach function spaces. In particular, we introduce the classes $M_\E$ of weights $w$ such that the Hardy operator is bounded on $\E_w$. A characterisation of the class $M_\E$ for general rearrangement invariant Banach function spaces $\E$, for example in the spirit of Muckenhoupt's condition, is, to the best of our knowledge, not given in the literature. However, in Section \ref{sec.weights} we show that $A_{p_\E}^-\subseteq M_\E$ for every rearrangement invariant Banach function space $\E$ with Boyd indices $p_\E$, $q_\E\in(1,\infty)$; see 
Theorem \ref{pq operators} and Corollary \ref{A-p}. The class $A_{p_\E}^-$ has been defined in \cite{ChKr15} and is a rather large class of weights for which an extrapolation result for singular integral operators (in the spirit of Cruz-Uribe, Martell \& P\'erez \cite{CrMaPe11} and Curbera et al. \cite{CGMP06}) as well as an extrapolation result for maximal regularity of Cauchy problems holds true. In particular, this class contains all restrictions of Muckenhoupt weights to the positive half-line, as well as all positive, decreasing functions on $(0,\infty )$. 

The connection to the first order Cauchy problem is shortly described in Section \ref{sec.appl}. Knowing the precise time regularity of solutions, depending on the initial values and right-hand sides, is fundamental for applications to fully nonlinear evolution equations. Potential applications of the general real interpolation spaces exist, however, also in the context of perturbation theory for abstract Cauchy problems like, for example, in Haak, Haase \& Kunstmann \cite{HaHaKu06}, or in the study of nonautonomous Cauchy problems, namely when regularity of the underlying operator family in some interpolation space is part of the assumptions for proving existence and uniqueness of solutions; see Yagi \cite{Ya90}. 

It would be desirable to have an inner description of the real interpolation spaces in the case when the sectorial operators are concrete partial differential operators and their domains are concrete function spaces. In the case when these operators act on $L^p$ spaces and when their domains are classical Sobolev spaces, we think of characterizations of the resulting Besov spaces in terms of integrability properties of functions and their translates; see Triebel \cite[\S 3.4.2, p.208]{Tr83}. Let us mention, however, that sometimes such a description amounts to identify trace spaces of more complicated, anisotropic, weighted Sobolev spaces, a problem which has only recently been solved in the $L^p$-setting with power weights; see Meyries \& Schnaubelt \cite{MeSc12,MeSc12a} and Meyries \& Veraar \cite{MeVe14}. 

\section{Generalised real interpolation spaces} \label{sec.definition}

In this section we review the real interpolation functors of the trace method in the context of general Banach function spaces $\Phi$ instead of weighted $L^p$ spaces. We give a short proof of the fact that we are really dealing with interpolation functors, and then we show that the trace functors coincide with the real interpolation functors of the $K$-method if the Hardy operator is bounded on $\Phi$.

Throughout, all function spaces -- with the exception of spaces of holomorphic functions -- are function spaces over the domain $\R_+ := [0,\infty)$. We denote by $\cM$ the set of all complex measurable functions on $\R_+$ equipped with the topology of convergence in Lebesgue measure, and we let $L^1_{loc}$ be the space of all locally integrable functions on $\R_+$. 

Following \cite{BrKr91}, a Banach space $\Phi$ is called a {\it Banach function space (over $\R_+$)} if it is both a Banach lattice and an order ideal of $\cM$. This means that $\Phi$ is a sublattice of $\cM$, and $|f|\leq |g|$, $g\in \Phi$, $f\in\cM$, implies $f\in \Phi$ and $\| f\|_\Phi \leq \| g\|_\Phi$. This definition is consistent with Meyer-Nieberg \cite{MN91}, for instance, where it is used in the sense of complete K\"othe function space, but it differs from the definition in Bennett \& Sharpley \cite{BeSh88}, where in addition a Fatou property is required to hold. Note that in Brudny{\u\i} \& Krugljak \cite{BrKr91}, Banach function spaces are called Banach lattices.  

\subsection{The trace method}

This interpolation method is relevant for the applications to Cauchy problems. 
Let $\Phi$ be a Banach function space which embeds continuously into $L^1_{loc}$.
Given a Banach space $X$, we define the vector-valued variant of $\Phi$ by
\[
\Phi (\X ) :=\left\{ f: \R_+ \rightarrow \X \;\;{\rm{measurable }}:  |f|_X \in \Phi \right\} .
\]
This space is a Banach space for the norm $\|f\|_{\Phi(X)}:=\| |f|_X \|_{\Phi}$ (for simplicity, as it is usually done, we identify the space of measurable functions with the space of equivalence classes of measurable functions).

Let $(X,Y)$ be an interpolation couple of Banach spaces. Consider the space 
\[
W^{1,\Phi}(X,Y):= \left\{u\in W^{1,1}_{loc} (X+Y) : u\in \Phi(Y),
 \dot u\in \Phi(X) \right\} .
\]
which is, thanks to the continuous embedding $\Phi \subseteq L^1_{loc}$, a Banach space for the natural norm. 

Then we define the corresponding {\em trace space}  
\[
T_{\Phi}(X,Y):= [X,Y]_{\Phi}:= \left\{ u(0): u\in W^{1,\Phi}(X,Y) \right\} 
\]
which is a Banach space for the quotient norm
\[
|x|_{[X,Y]_\Phi} := \inf \left\{\|u \|_{\Phi(Y)} + \|\dot u \|_{\Phi(X)}: u \in W^{1,\Phi}(X,Y), u(0)=x \right\}.
\]
We say ``quotient norm'' because 
\[
W^{1,\Phi}_0 (X,Y) := \{ u\in W^{1,\Phi}(X,Y) : u(0)=0\}
\]
is a closed subspace of $\mrew$ and the trace space equipped with the above norm is isometrically isomorphic to the quotient space $W^{1,\Phi}(X,Y) / W^{1,\Phi}_0 (X,Y)$. \\

Our first lemma shows that $T_\Phi$ is an interpolation functor if and only if $\chi_{(0,1)} \in \Phi$.

\begin{lemma} \label{T-lemma}
Let $\Phi$ be a Banach function space such that $\Phi\subseteq L^1_{loc}$ continuously, and let $(X,Y)$ be an interpolation couple of Banach spaces.
\begin{itemize}
 \item [(i)] If $[X,Y]_{\Phi} \not= \{ 0\}$, then $\chi_{(0,\tau)}\in \Phi$ for some $\tau >0$.
 \item [(ii)] If $\chi_{(0,\tau )}\in \Phi$ for some $\tau >0$, then $[X,Y]_{\Phi}$ is an interpolation space relative to $(X,Y)$.
\end{itemize}
\end{lemma}

\begin{proof}
(i) Let $x\in [X,Y]_{\Phi}$. By definition, there exists $u\in W^{1,1}_{loc} (X+Y)\cap \Phi (Y)$ such that $u(0)=x$. If $x\not= 0$, then, by continuity of $u$ with values in $X+Y$, there exist $\tau >0$, $c \geq 0$ such that
\[
  \chi_{(0,\tau )} \leq c\,|u|_{X+Y} \leq c\, |u|_{Y} \quad \text{a.e. on } (0,\infty ).
\]
This gives $\chi_{(0,\tau )}\in\Phi$.\\
(ii) First we show that, if $\chi_{(0,\tau )}\in \Phi$, then $X\cap Y \subseteq [X,Y]_{\Phi}\subseteq X+Y$ with continuous embeddings. Assume that $\chi_{(0,\tau )}\in\Phi$ for some $\tau >0$. Choose any function $\varphi\in C^1_c (\R_+ )$ such that $\varphi (0) = 1$ and $\supp \phi\subseteq [0,\tau ]$. Then for every $x\in X\cap Y$ one has $\varphi (\cdot ) \, x \in W^{1,\Phi} (X,Y)$, which implies $x\in [X,Y]_{\Phi}$ and $|x|_{[X,Y]_{\Phi}} \leq \| \varphi\|_{\Phi} \, |x|_Y + \| \dot\varphi \|_{\Phi} \, |x|_X \leq C\, |x|_{X\cap Y}$. For the second inclusion, since $\Phi$ is continuously embedded into $L^1_{loc}$, there exists a constant $C\geq 0$ such that
\[
 \| f \chi_{[0,\tau ]} \|_{L^1} \leq C \, \| f\chi_{[0,\tau ]} \|_{\Phi} \quad \text{for every } f\in\Phi. 
\]

Let $x\in [X,Y]_{\Phi}$. Let $u\in \mrew$ such that $u(0)=x$. Note that $x=x-u(t) + u(t) = - \int_0^t \dot u \, \ud s + u(t)$ ($t>0$). 
Consequently, 
\begin{eqnarray*}
 |x|_{X+Y} &\leq& |\int_0^t \dot u \, \ud s|_X + |u(t)|_Y\\
 &\leq& C \|\dot u\|_{\Phi(X)} + |u(t)|_Y \quad (t\in (0,\tau )).
\end{eqnarray*}
Therefore, 
\[
 \|\chi_{(0,\tau )}\|_{\Phi}\, |x|_{X+Y}  \leq C \|\dot u\|_{\Phi(X)}\, \|\chi_{(0,\tau )}\|_{\Phi} + \|u\|_{\Phi(Y)}, 
\]
which gives the desired claim.
 
Finally, in order to prove the interpolation property of $[X,Y]_\Phi$, let $T$ be an admissible operator with respect to the couple $(X,Y)$. Then $T$ induces in a natural way the bounded multiplication operator, again denoted by $T$, on $\mrew$ given by
\[
 (Tu) (t) := T u(t) \qquad (u\in\mrew , \, t\in\R_+ ) .
\]
This multiplication operator leaves the subspace $W^{1,\Phi}_0 (X,Y)$ invariant and thus induces a bounded operator on the quotient space $\mrew /W^{1,\Phi}_0 (X,Y)$. With the natural identification, this multiplication operator on the quotient space coincides with the original operator $T$. 
\end{proof}

\subsection{The $K$-method}

This section closely follows \cite{BrKr91}. Again, let $(X,Y)$ be an interpolation couple of Banach spaces and let $\Phi$ be a Banach function space.
Recall first the definition of the {\em $K$-functional}, that is, for every $x\in X+Y$ and $t>0$ put
\[
K(t,x):=\inf\left\{ |a|_X+t|b|_Y: a\in X, b\in Y, a+b=x \right\}.
\] 
With the help of the $K$-functional, we define the space
\[
K_\Phi(X,Y):= (X,Y)_{\Phi}:= \left\{ x \in X+Y: \; [(0,\infty)\ni t \mapsto t^{-1}K(t,x) ]\in \Phi 
\right\} ,
\]
and equip it with the norm
\[
|x|_{(X,Y)_\Phi}:=\left\| (\cdot)^{-1}K(\cdot, x) \right\|_\Phi \quad \quad \quad (x\in {(X,Y)_\Phi}).
\]
This definition of real interpolation spaces differs very slightly from the one in \cite[Section 3.3]{BrKr91}, where $\Phi$ runs through Banach function spaces containing the function $\min(1,\cdot)$, and $x\in K_\Phi(X,Y)$ iff $K(\cdot, x)\in \Phi$. The point is only the factor $\frac{1}{t}$ between the two definitions. The requirement $\min(1,\cdot)\in\Phi$ in the setting of \cite{BrKr91} is natural as the following lemma shows. In fact, $K_\Phi$ is an interpolation functor if and only if $\min(1,(\cdot)^{-1})\in \Phi$; cf. \cite[Proposition 3.3.1]{BrKr91}.
 
\begin{lemma}\label{K-lemma}
Let $\Phi$ be a Banach function space. Then the following conditions hold.
\begin{itemize}
 \item [(i)] 
If $(X,Y)_{\Phi} \not= \{ 0\}$, then $\min(1,(\cdot)^{-1})\in\Phi$.
 \item [(ii)]  If 
$\min(1,(\cdot)^{-1})\in\Phi$, then $(X,Y)_{\Phi}$ is an interpolation space relative to $(X,Y)$. 
\end{itemize}
\end{lemma}

The proof is standard, we include details for the convenience of the reader.

\begin{proof} $(i)$ Note that for every $x\in X+Y$ we have 
\begin{equation}\label{K-ineq}
 \min(1,t^{-1}) \, |x|_{X+Y} \leq t^{-1}K(t,x) \quad (t>0).
\end{equation}
Therefore, the claim follows from the order ideal property of $\Phi$.\\ 
$(ii)$ First note that $X\cap Y\subseteq (X,Y)_{\Phi} \subseteq X+Y$ continuously. Indeed, the first embedding follows directly from   
\[
t^{-1}K(t,x)\leq \min(1,t^{-1}) \, |x|_{X\cap Y} \quad (t>0, x\in X\cap Y),
\] and the second one, from \eqref{K-ineq} above.

Let us show that $(X,Y)_{\Phi}$ is a Banach space. Let $\sum_n x_n$ be an absolutely convergent series in $(X,Y)_{\Phi}$. We have to show that this series is convergent. By definition of the norm in $(X,Y)_{\Phi}$, the series $\sum_n (\cdot )^{-1} K(\cdot ,x_n)$ is absolutely convergent in $\Phi$, and by \eqref{K-ineq}, the series $\sum_n x_n$ is absolutely convergent in $X+Y$. The spaces $\Phi$ and $X+Y$ being complete, the preceding two series converge to some elements $\phi\in\Phi$ and $x\in X+Y$, respectively. Since, for every $t>0$, $K(t, \cdot)$ is a norm on $X+Y$, the triangle inequality yields, for every $n\in\N$,  
\[
(\cdot)^{-1}K(\cdot, x - \sum_{m=1}^n x_m) \leq \sum_{m>n} (\cdot)^{-1}K(\cdot, x_m ) \text{ everywhere on } (0,\infty ).
\]
By \cite[Lemma 3.3.2]{BrKr91}, the series $\sum_n (\cdot )^{-1} K(\cdot ,x_n)$ converges pointwise almost everywhere to $\phi$, and hence the right-hand side of the inequality above can be estimated by $\phi$. The ideal property of $\Phi$ thus implies $x\in (X,Y)_{\Phi}$. The convergence of $\sum_n x_n$ to $x$ in $(X,Y)_{\Phi}$ then follows from the above inequality and absolute convergence of the series $\sum_n (\cdot )^{-1} K(\cdot ,x_n)$ in $\Phi$.

Finally, let $T$ be an admissible operator relative to $(X,Y)$. We show that $T$ is bounded on $(X,Y)_{\Phi}$ with the operator norm 
\begin{equation}\label{interp norm}
 |T|_{\cL((X,Y)_{\Phi})} \leq \max(|T|_{\cL(X)},|T|_{\cL(Y)}).
\end{equation}
If $x\in (X,Y)_{\Phi}$, then for every $a\in X$ and $b\in Y$ such that $a+b=x$, we have that 
\[
|Ta|_X + t|Tb|_Y \leq
 \max(|T|_{\cL(X)},|T|_{\cL(Y)}) \, \left( |a|_X + t|b|_Y \right), \quad \quad (t>0).
\]
Consequently, 
\[
K(t, Tx)\leq \max(|T|_{\cL(X)},|T|_{\cL(Y)}) \, K(t, x), 
\]
which completes the proof. 
\end{proof}

\subsection{Equivalence of trace method and $K$-method}
Let $p\in [1,\infty )$ and $\theta\in (0,1)$. In the case $\Phi = L^p(t^{p(1-\theta) - 1}\ud t)$ the spaces $[X,Y]_\Phi$ and $(X,Y)_\Phi$ both yield the classical real interpolation space which is usually denoted by $(X,Y)_{\theta,p}$; see, for example, \cite[Proposition 1.13]{Lu09}, \cite[Theorem 1.8.2, p.~44]{Tr95}. Our next aim is to show that the equality $[X,Y]_\Phi = (X,Y)_\Phi$ still holds under rather 'general' assumptions on $\Phi$.

Let $P$ be the {\em Hardy operator} and $Q$ its (formal) adjoint, that is, the integral operators given by 
\[
Pf(t):=\frac{1}{t}\int_0^t f (s)\,\ud s, \quad\quad  Qf(t):= \int_t^\infty f (s)\,\frac{\ud s}{s}
\]
for all $f\in\cM$ such that the respective integrals exist for all $t\in (0,\infty )$. The maximal domain of $P$ is hence the space $L^1_{loc}$, while the maximal domain of $Q$ is the set of all measurable functions $f\in\cM$ such that $f\chi_{(\tau ,\infty )} \in L^1 ( t^{-1}\, \ud t)$ for every $\tau >0$. By changing the order of integration one obtains 
\begin{equation*} 
 \int_0^\infty f Qg \,\ud t = \int_0^\infty  g Pf\,\ud t
\end{equation*}
for all $f$, $g\in\cM$ such that $Pf$, $Qg$ and the above integrals exist. \\

Let $\mathfrak{L}_P$ (resp. $\mathfrak{L}_{P+Q}$) denote the class of Banach function spaces $\Phi$ such that the Hardy operator $P$ (resp. the Calder\'on operator $P+Q$) is bounded on $\Phi$. Note that if $\Phi \in \mathfrak{L}_P$, then $\Phi \subseteq L_{loc}^1$, and $\chi_{(0,1)}\in \Phi$ if and only if $\min(1,\frac{1}{(\cdot)})\in \Phi$. Moreover, by Lemmas \ref{T-lemma} (i) and \ref{K-lemma} (i), if $\chi_{(0,1)}\notin \Phi$, then $[X,Y]_\Phi =\{0\}= (X,Y)_\Phi$. 

The following lemma shows that the functors $T_\Phi$ and $K_\Phi$ coincide for $\Phi \in \mathfrak{L}_P$.

\begin{theorem}\label{trace=interp}
For every Banach function space $\Phi\in \mathfrak{L}_P$ and every interpolation Banach couple $(X,Y)$, $[X,Y]_\Phi = (X,Y)_\Phi$. 
\end{theorem}

\begin{proof}[Proof of Theorem \ref{trace=interp}] Let $(X,Y)$ be an interpolation Banach couple. 
Fix $x\in (X,Y)_\Phi$. We show that there exists $u\in \mrew$ such that $u(0)=x$. For each $n\in \N$ there exist $a_n\in X$ and $b_n\in Y$ such that 
\[
a_n+b_n=x,\quad\quad |a_n|_{X} + n^{-1} |b_n|_Y \leq 2 K(n^{-1}, x).
\]
Set 
\[
v(t):=\sum_{n=1}^\infty b_{n+1} \chi_{(\frac{1}{n+1}, \frac{1}{n}]}(t),\quad\quad 
u(t):=\frac{1}{t}\int_0^t v  \, \ud s\quad (t>0).
\]
By monotonicity, $\lim_{t\rightarrow 0^+}K(t,x)$ exists, and since $\Phi \subseteq L^1_{loc} (\R_+)$ continuously, we find necessarily that $\lim_{t\rightarrow 0^+} K(t,x) = 0$. Thus, $\lim_{n\to\infty} |a_n|_X = 0$, and  $|x-b_n|_{X+Y}\leq 
|a_n|_{X} \rightarrow 0$ as $n\rightarrow \infty$. Consequently, $x=\lim_{t\rightarrow 0^+} v(t) = \lim_{t\rightarrow 0^+} u(t)$ in $X+Y$. Next, for every $t>0$,
\[
|v(t)|_Y\leq 2\, \sum_{n=1}^\infty (n+1) K\left((n+1)^{-1},x\right) \chi_{(\frac{1}{n+1},\frac{1}{n}]}(t)  \leq 4t^{-1} \, K(t, x) .
\]
Therefore, $|v(\cdot)|_Y\in \Phi$, and $\|v\|_{\Phi(Y)}\leq 4|x|_{(X,Y)_{\Phi}}$. 
Since, $|u(\cdot )|_Y\leq P(|v(\cdot)|_Y)$ on $(0,\infty )$, by the definition of $\mathfrak{L}_P$, we get $u\in {\Phi(Y)}$ and 
$\|u\|_{\Phi(Y)}\leq 4\|P\|_{\cL (\Phi)} |x|_{(X,Y)_{\Phi}}$.

Let $g(t):=\sum_{n=1}^\infty a_{n+1} \chi_{(\frac{1}{n+1},\frac{1}{n}]}(t)$. Then $u(t)=x- \frac{1}{t} \int_0^t g \,\ud s$, and 
\[
\dot u(t)= \frac{1}{t^2}\int_0^t g \, \ud s - \frac{1}{t}g(t),
\] 
where  the derivative exists for almost every $t>0$ in $X$. Since $|g(t)|_X\leq 2\, K(t, x)$, we get 
\[
|\dot u(t)|_X \leq t^{-1} \, [ \sup_{0<s<t}|g(s)|_X + |g(t)|_X ] \leq 4 t^{-1}\, K(t,x)\quad (t>0).
\]
Consequently, $\|\dot u\|_{\Phi(X)}\leq 4 \, |x|_{(X,Y)_{\Phi}}$. Thus, $x\in [X,Y]_{\Phi}$ with $$|x|_{[X,Y]_{\Phi}}\leq 4\max(\|P\|_{\cL (\Phi )},1) \, |x|_{(X,Y)_{\Phi}}.$$

To show the converse inclusion, let $x\in [X,Y]_{\Phi}$. Let $u\in \mrew$ such that $u(0)=x$. Note that $x=x-u(t) + u(t) = - \int_0^t \dot u \, \ud s + u(t)$ ($t>0$). 
Consequently, 
\[
K(t,x)\leq \left|\int_0^t \dot u \,\ud s\right|_X + t|u(t)|_Y\quad \quad (t>0).
\] 
Therefore, we get  
$|x|_{(X,Y)_{\Phi}} \leq \|P\|_{\cL (\Phi )} \, \|\dot u\|_{\Phi(X)} +
\| u\|_{\Phi(Y)}$, that is, $x\in (X,Y)_{\Phi}$. The proof is complete.
\end{proof}

\begin{remark} \label{rem.hardy}
(a) Following \cite{LiTz79}, we define the {\em lower} and {\em upper Boyd indices} of a Banach function space $\Phi$ respectively by 
\begin{equation*}
p_\Phi  := \lim_{t\to \infty} \frac{\log t}{\log h_\Phi (t)}  \quad \text{ and }\quad 
q_\Phi  := \lim_{t\to 0+} \frac{\log t}{\log h_\Phi (t)},
\end{equation*}
where $h_{\Phi}(t)=\|D_t\|_{\cL (\Phi )}$ and $D_t: \Phi\rightarrow\Phi$ ($t>0$) is the {\em dilation operator} defined by 
\[
D_tf(s)=f(s/t), \qquad 0<t<\infty, \quad f\in \Phi .
\]
Of course, this definition of the Boyd indices is admissible only if the dilation operators are well defined and bounded on $\Phi$. One always has $1\le p_\Phi\le q_\Phi\le\infty$; see, for example, \cite[Proposition 5.13, p.~149]{BeSh88}, where the Boyd indices are defined as the reciprocals with respect to our definitions. The Boyd indices can be computed explicitly for many examples of concrete rearrangement invariant Banach function spaces. For example, for the Lebesgue spaces $\Phi = L^p$ or the Lorentz spaces $\Phi = L^{p,q}$ one has $p_\Phi = q_\Phi = p$; see \cite[Chapter 4]{BeSh88}. We refer the reader primarily to the monographs by Bennett \& Sharpley \cite{BeSh88} and Lindenstrauss \& Tzafriri \cite{LiTz79} for more background on rearrangement invariant Banach function spaces. 

By \cite[Theorem 5.15, p.~150]{BeSh88}, the Hardy operator $P$ (resp. $P+Q$) is bounded on a rearrangement invariant Banach function space $\Phi$ if and only if $p_\Phi>1$ (resp. $1<p_\Phi \leq q_\Phi<\infty$). A close analysis of the proof of this result shows that the condition $p_\Phi>1$ actually implies $\Phi\in\mathfrak{L}_P$ for all Banach function spaces $\Phi$ with norming associate space $\Phi'$. By the Lorentz-Luxemburg theorem \cite[Theorem 2.7, p.10]{BeSh88}, 
every Banach function space with Fatou property has a norming associate space.

(b) An interpolation space $Z$ with respect to the interpolation couple $(X,Y)$ is a {\em $K$-monotone interpolation space} if there exists a constant $C\geq 0$ such that the inequality $K(\cdot ,y)\leq K(\cdot ,x)$ for $x\in Z$ and $y\in X+Y$ implies that $y\in Z$, and $|y|_{Z}\leq C \, |x|_{Z}$. By \cite[Theorem 3.3.20]{BrKr91}, for any $K$-monotone interpolation space $Z$ there exists a Banach function space $\Phi$ such that 
\[ 
Z= \{ x\in X: (\cdot )^{-1} K(\cdot, x)\in \Phi \}
\] 
and $|x|_{Z}\sim \| (\cdot )^{-1} K(\cdot, x)\|_{\Phi}$ ($x\in Z$). It is natural to ask whether for a general Banach function space $\Phi$, which embeds continuously into $L^1_{loc}$ and for which $\min(1, (\cdot)^{-1})\in \Phi$, the trace space $[X,Y]_{\Phi}$ is a $K$-monotone interpolation space, too. 
\end{remark}

\section{The functional calculus representation of generalised real interpolation spaces} \label{sec.functional.calculus}

Let $X$ be a Banach space. A linear, not necessarily densely defined operator $A$ on $X$ is {\em sectorial of angle $\phi\in (0,\pi )$} if $\sigma(A)\subseteq \bar S_\phi$ and $\sup_{\lambda\in \C\setminus S_{\phi'}}|\lam R(\lambda, A)|_{\cL(X)} <\infty$ for every $\phi'\in(\phi,\pi)$. Here, $S_\phi$ stands for the open sector $\{z\in \C: z\neq 0, |\arg z|<\phi\}$. We simply say that $A$ is {\em sectorial} if it is sectorial of some angle $\phi\in (0,\pi )$. For a sectorial operator $A$ we define the {\em angle of sectoriality} by $\phi_A:=\inf \bigl\{\phi>0: A$ is sectorial of angle $\phi \bigr\}$.

In this section, we provide an extension of the functional calculus representation of the classical real interpolation spaces $(X, \dom{A})_{\theta,q}$ to the real interpolation spaces $(X, \dom{A})_{\Phi}$. As in the case of the results from the previous section, the main ingredient in the proof of this representation is the boundedness of the Hardy operator $P$ and its adjoint. We refer the reader to Haase \cite{Hs06} for the background on the holomorphic functional calculus of sectorial operators. In fact, following \cite{Hs06}, let $H^\infty(S_\phi)$ ($\phi\in (0,\pi)$) be the algebra of all bounded, holomorphic functions on $S_\phi$, and consider the following subalgebras of $H^\infty(S_\phi)$: 
\begin{align*}
 H^\infty_0(S_\phi) & := \{ f\in H^\infty(S_\phi): \exists C, \, s>0 \text{ s.t. } |f(z)|\leq C \min (|z|^s, |z|^{-s})\;(z\in S_\phi) \} , \\
 \cE(S_\phi) & := H^\infty_0(S_\phi) \oplus \big< (1+z)^{-1}\bigr> \oplus \bigl< 1 \bigr >.
\end{align*}
Assume that $A$ is sectorial. Then for every $f\in H^\infty_0 (S_{\phi} )$ ($\phi\in (\phi_A ,\pi )$) one defines
\[
 f(A) := \frac{1}{2\pi i} \int_{\partial S_\phi} f(z) R(z,A) \; \ud z ,
\]
and if $f = f_0 + \frac{\lambda}{1+\cdot} + \mu \in \cE (S_\phi )$ ($f_0\in H^\infty_0 (S_\phi )$, $\lambda$, $\mu\in\C$), then
\[
 f(A) := f_0 (A) + \lambda \, (I+A)^{-1} + \mu .
\]
In this way one obtains an algebra homomorphism $\cE (S_\phi ) \to \cL (X)$, $f\mapsto f(A)$, which is called the $\cE$-primary (holomorphic) functional calculus. 

In \cite[Theorem 6.2.1]{Hs06}, Haase proved the following generalisation of a classical result due to Komatsu: \emph{for every $\Re \alpha>0$ 
\[
{(X,\dom{A^\alpha})_{\theta, q}}= \left\{ x\in X: \left( \int_0^\infty |t^{-\theta}\psi(tA)x|^q_X \frac{\ud t}{t}\right)^{1/q}<\infty\right\}
\] 
and the norm $|\cdot|_{(X,\dom{A^\alpha})_{\theta,q}}$ is equivalent to $|\cdot|_X+ \left( \int_0^\infty |t^{-\theta}\psi(tA) \cdot |^q_X \frac{\ud t}{t}\right)^{1/q}$.} Here, $\psi$ is a bounded holomorphic function satisfying some growth conditions depending on $\alpha$. For instance, $\psi(z)= z^\alpha (1+z)^{- \alpha}$ is an appropriate function and the choice of this particular function corresponds to Komatsu's original result \cite{Ko67,Ko72}. 

The following result extends the above characterisation to general real interpolation spaces. Recall that $\mathfrak{L}_P$ denotes the class of Banach function spaces $\Phi$ such that the Hardy operator $P$ is bounded on $\Phi$.

\begin{theorem}\label{fcrepresent} 
Let $A$ be sectorial and $\phi \in (\phi_A, \pi)$. Assume that $\psi$ is a holomorphic function on $S_\phi$ such that the following properties hold:
\begin{itemize}
 \item [(i)] $\psi$, $(\cdot)^{-1}\psi\in \cE(S_\phi)$; 
\item [(ii)] $\lim_{z\rightarrow 0}z^{-1}\psi(z)\neq 0$ and $\psi(z)\neq 0$ for all $z\in S_\phi$, and 
\[ 
\sup_{z\in S_\phi, s\geq 1} \left| \frac{\psi(sz)}{s \psi(z)} \right|<\infty.
\]
\end{itemize}

 Then, for every $\Phi \in \mathfrak{L}_P$ with $\min(1,\frac{1}{(\cdot)})\in \Phi$, 
\[ 
{(X,\dom{A})_\Phi}= \bigl\{ x\in X: \left[(0,\infty)\ni t \mapsto |t^{-1}\psi(tA)x|_X\, \right] \in \Phi \bigr\}
\] 
and  
\[
|x|_{(X,\dom{A})_\Phi} \simeq |x|_X+ \bigl\|(\cdot)^{-1}|\psi(\cdot A)x|_X\bigr\|_{\Phi}\quad  \bigl(x\in (X,\dom{A})_{\Phi}\bigr).
\]
If, in addition, $A$ is invertible, then 
\[
|x|_{(X,D_A)_\Phi} \simeq \left\|(\cdot)^{-1}|\psi(\cdot A)x|_X\right\|_{\Phi}\quad (x\in (X,\dom{A})_\Phi).
\]
\end{theorem}

\begin{remark} \label{rem.kk}
Theorem \ref{fcrepresent} extends a result by Kalton and Kucherenko in \cite{KaKu10}; see \cite[Theorem 5.4]{KaKu10} (for $\sigma=0, \tau = 1$). Our approach differs from the one in \cite{KaKu10} in a few points. First, our notion of sectorial operator $A$ does not require $A$ to be one-to-one, nor does it require that the domain $\dom{A}$ and the range $\rg{A}$ are dense in $X$. Second, our assumption on the Banach function space $\Phi$ does not require that the space $L^\infty_{c}$ of essentially bounded functions with compact support in $(0,\infty )$ is dense in $\Phi$, nor do we impose a restriction on the upper Boyd index. The assumption on the lower Boyd index of the Banach function space $\Phi$ corresponds to our assumption on the boundedness of the Hardy operator on $\Phi$; compare with Remark \ref{rem.hardy} (a).
 
Note that $L^\infty_{c}$ is not dense in the Lorentz space $\Phi:=L^{p,\infty}$. If one denotes by $\Phi_0$ the closure of $L^\infty_{c}$ in $\Phi$, then one can show that $(X,Y)_{\Phi_0}$ is a proper subset of $(X,Y)_{\Phi}$.\\
\end{remark}

The proof of Theorem \ref{fcrepresent} follows essentially the lines of the proof of \cite[Theorem 6.2.1]{Hs06} (the case $\Re \alpha =1$). The new point of our approach is to replace the Hardy-Young inequality (used, for example, in \cite[Lemma 6.2.6]{Hs06}) by the boundedness of the Hardy operator $P$ on $\Phi$. For the convenience of the reader we sketch the proof of Theorem \ref{fcrepresent}.

\begin{proof}
We first show the inclusion ``$\supseteq$''. By \cite[Lemma 6.2.7]{Hs06}, there exist functions $f\in H^\infty_0(S_\phi)$ and $g\in \cE(S_\phi)$ such that 
\[
\int_0^1(f\psi)(sz) \, \frac{\ud s}{s} + g(z)\psi(z)z^{-1} = 1\quad (z\in S_\phi).
\]
Set 
\[
h_1(z) :=\int_0^1(f\psi)(sz) \, \frac{\ud s}{s} \quad \text{and} \quad h_2(z) := g(z)\psi(z)z^{-1} \quad  (z\in S_\phi). 
\]
Note that $h_1$, $h_2\in \cE(S_\phi)$ (\cite[Example 2.2.6]{Hs06}), and $x = h_1 (tA)x + h_2(tA)x$ for every $x\in X$ and every $t>0$. By algebraic properties of the $\cE$-primary functional calculus, see for example \cite[Theorem 2.3.3]{Hs06},
\[
h_2(tA)x\in \dom{A}\quad \textrm{and} \quad Ah_2(tA)x = t^{-1} g(tA)\psi(tA)x \quad (x\in X, \, t>0).
\] 
Thus, on the one hand,
\[
 K(t,x) \leq | h_1 (tA)x|_X + t \, | h_2 (tA)x|_{\dom{A}} \quad (x\in X, \, t>0 ) .
\]
On the other hand, by \cite[Proposition 2.6.11]{Hs06}, there exists a constant $C\geq 1$ such that 
\begin{align*}
|h_2(tA)x|_{X} & \leq C\, |x|_X , \text{ and} \\
|Ah_2(tA)x|_{X} & \leq C \, t^{-1}|\psi(tA)x|_X  \quad  (x\in X, \, t>0).
\end{align*}
Since $|(f\psi)(t A)x|_X \leq C |\psi(t A)x|_X$, see again \cite[Proposition 2.6.11]{Hs06}, one has
\begin{align*}
 t^{-1}|h_1(tA)x|_X & = \left|\frac{1}{t} \int_0^1 (f\psi)(stA)x \frac{\ud s}{s}\right|_X \\
 & \leq \frac{1}{t}\int_0^t | (f\psi)(sA)x |_X \frac{\ud s}{s}\\ 
 & \leq  C \, P\bigl( (\cdot)^{-1}|\psi(\cdot A)x|_X \bigr)(t) \quad \quad (x\in X, \, t>0).
\end{align*} 
Combining the preceding estimates, and since $K(t,x)\leq |x|_X$, we obtain 
\begin{equation} \label{eq.est2}
t^{-1} K(t,x)\leq C\, \bigl( P\bigl( (\cdot)^{-1}|\psi(\cdot A)x|_X \bigr)(t) + t^{-1}|\psi(tA)x|_X +  \min(1,t^{-1}) \, |x|_X \bigr) 
\end{equation}
for every $x\in X$ and every $t>0$. This estimate and Lemma \ref{lem.chi} yield the inclusion ``$\supseteq$'' and the corresponding estimate for the norm. 

To prove the converse inclusion ``$\subseteq$'', note that $\psi(tA)x=\psi(tA)a + t\gamma(tA)Ab$ for every $x=a+b$ with 
$a\in X$ and $b\in \dom{A}$, where $\gamma(z):=\psi(z)z^{-1}$ ($z\in S_\phi$). Hence, by \cite[Proposition 2.6.11]{Hs06}, $|\phi(tA)x|_X \leq C\bigl(|a|_X + t|b|_{\dom{A}}\bigr)$  ($t>0$). Since $\min(1,t^{-1})\, |x|_X \leq t^{-1}K(t,x)$ ($x\in X$) and $\min(1,t^{-1})\in \Phi$ (see Lemma \ref{lem.chi}), this proves the converse inclusion.

If, in addition, $A$ is invertible, then $|\cdot |_{\dom{A}} \simeq |A\cdot |_X$, and therefore the last term on the right-hand side of \eqref{eq.est2}, which comes from an estimate of $|h_2 (tA)x|_X$, can be dropped. In that case, we obtain $|x|_{(X,\dom{A})_\Phi} \leq C\, \left\|(\cdot)^{-1}|\psi(\cdot A)x|_X\right\|_{\Phi}$. The other inequality is proved as above. 
\end{proof}

\begin{remark} \label{holomorphic}
One can easily check that the function $\psi(z):= z e^{-z}$ satisfies the assumptions of Theorem \ref{fcrepresent} for every $\phi \in (0, \pi/2)$. Hence, if $-A$ is the generator of a bounded holomorphic $C_0$-semigroup on $X$, then we get the representation
 \begin{equation}\label{holrep}
  {(X,\dom{A})_{\Phi}} = \bigl\{ x\in X: \left[(0,\infty)\ni t \mapsto |Ae^{-tA}x|_X\, \right] \in \Phi \bigr\}
 \end{equation}
and 
\[
|x|_{(X,\dom{A})_\Phi} \simeq |x|_X + \bigl\| Ae^{-\cdot A}x\bigr\|_{\Phi(X)}\quad \quad \quad \bigl(x\in (X,\dom{A})_{\Phi}\bigr) 
\]
for any $\Phi \in \mathfrak{L}_P$. For the classical real interpolation spaces $(X, \dom{A})_{\theta, q}$ ($\theta\in (0,1)$, $q\in [1, \infty]$), this representation is well-known; see Komatsu \cite{Ko67} or Butzer \& Berens \cite{BuBe67}. We point also out that 
if an operator $A$ has $L^p$-maximal regularity, then $-A$ generates a holomorphic $C_0$-semigroup.
\end{remark}

Our next result shows that under additional assumptions on the Banach function space $\Phi$ one can get the conclusion of Theorem \ref{fcrepresent} for all functions $\psi$ satisfying merely the condition (i) of Theorem \ref{fcrepresent}; cf. \cite[Theorem 6.2.9]{Hs06}. Recall that $\mathfrak{L}_{P+Q}$ stands for the class of Banach function space $\Phi$ with $\min(1,\frac{1}{(\cdot)})\in \Phi$ and such that the Calder\'on operator $P+Q$ is bounded on $\Phi$.

\begin{theorem}\label{fcrepresent2} 
Let $A$ be sectorial and $\phi\in (\phi_A, \pi)$. Let $\psi\neq 0$ be any holomorphic function on $S_\phi$ such that $\psi$, $(\cdot)^{-1}\psi \in \cE(S_\phi)$. 
Then, for any Banach function space $\Phi \in \mathfrak{L}_{P+Q}$ we have 
\[ 
(X,\dom{A})_{\Phi} = \left\{ x\in X: \left[(0,\infty)\ni t \mapsto |t^{-1}\psi(tA)x|_X\, \right] \in \Phi \right\}
\] 
and  
\[
|x|_{(X,\dom{A})_{\Phi}} \simeq |x|_X + 
\left\|(\cdot)^{-1}|\psi(\cdot A)x|_X\right\|_{\Phi}\quad \bigl(x\in (X,\dom{A})_{\Phi}\bigr).
\]
If, in addition, $A$ is invertible, then 
\[
|x|_{(X,D_A)_{\Phi}} \simeq \left\|(\cdot)^{-1}|\psi(\cdot A)x|_X\right\|_{\Phi}\quad (x\in (X,\dom{A})_{\Phi}).
\]
\end{theorem} 

The proof  follows the lines of the proof of \cite[Theorem 6.2.9]{Hs06}. We provide only the main supplementary observations which should be made.

\begin{proof}  
As in the proof of Theorem \ref{fcrepresent}, we first prove the inclusion ``$\supseteq$''. By \cite[Lemma 6.2.5]{Hs06}, there exists a function $f\in H^\infty(S_\phi)$ such that 
$\tilde f:=(\cdot) f\in H^\infty_0(S_\phi)$ and $\int_0^\infty (f\psi)(s) \frac{\ud s}{s} =1$. Define the functions $h_1$ and $h_2$ as follows: 
\[
h_1(z):=\int_0^1(f\psi)(sz) \frac{\ud s}{s} \quad \textrm{and}\quad h_2(z):=\int_1^\infty(f\psi)(sz) \frac{\ud s}{s}\quad (z\in S_\phi).
\]
Then, $h_1$, $h_2\in \cE(S_\phi)$  (see \cite[Example 2.2.6]{Hs06}), and $h=(\cdot)h_2\in H_0^\infty(S_\phi)$ with $h(A)=\int_1^\infty s^{-1}(\tilde f \psi)(sA) \frac{\ud s}{s}$; see \cite[Lemma 6.2.10]{Hs06}.
Therefore, for every $x\in X$ and $t>0$ we have $x= h_1(tA)x + h_2(tA)x$ with $h_2(tA)x\in \dom{A}$. Thus, 
\[
K(t, x)\leq |h_1(tA)x|_X + t|h_2(tA)x|_{\dom{A}} .
\] 
By \cite[Proposition 2.6.11]{Hs06}, there exists a constant $C\geq 1$ such that 
\begin{align*}
 |h_2(tA)x|_{X} &\leq C\, |x|_X, \\
 |(\tilde f \psi) (tA)x|_X & \leq C\, |\psi (tA)x|_X, \text{ and} \\ 
 |(f \psi) (tA)x|_X &\leq C\, |\psi (tA)x|_X \quad (x\in X,\, t>0).
\end{align*}
Therefore, 
\begin{align*}
 t^{-1}|h_1(tA)x|_X & \leq \frac{1}{t}\int_0^t | (f\psi)(sA)x |_X \frac{ds}{s} \\
 & \leq C \, \frac{1}{t}\int_0^t | \psi(sA)x |_X \frac{ds}{s}\\
 &\leq C \, P \bigl((\cdot)^{-1} |\psi(\cdot A) x|_X \bigr) (t) 
\intertext{and}
 |Ah_2(tA)x|_{X} 
 & = \left|\int_t^\infty s^{-1} (\tilde f \psi)(s tA) \frac{\ud s}{s}\right|_X \\
 & \leq C \, \int_t^\infty s^{-1} |\psi(tA) x|_X \frac{\ud s}{s} \\
 & \leq C \, Q \bigl((\cdot)^{-1} |\psi(\cdot A) x|_X \bigr)(t) .
\end{align*}
Combining the preceding estimates, and since $K(t,x)\leq |x|_X$, we obtain
\begin{equation} \label{eq.est1}
 t^{-1}K(t,x) \leq C \, \bigl( (P+Q) ((\cdot)^{-1} |\psi(\cdot A) x|_X \bigr) (t) + \min(1, t^{-1})|x|_{X} \bigr)
\end{equation}
for every $x\in X$ and $t>0$. This estimate gives the desired claim. A similar argument as in the proof of Theorem \ref{fcrepresent} yields the converse inclusion.

If, in addition, $A$ is invertible, then $|\cdot |_{\dom{A}} \simeq |A\cdot |_X$, and therefore the last term on the right-hand side of \eqref{eq.est1}, which comes from an estimate of $|h_2 (tA)x|_X$, can be dropped. In that case, we obtain $|x|_{(X,\dom{A})_{\Phi}} \leq C\, \left\|(\cdot)^{-1}|\psi(\cdot A)x|_X\right\|_{\Phi}$. The other inequality is proved as above. 
\end{proof}

A limit case occurs in the statement of Theorem \ref{fcrepresent} when $\phi_A = \frac{\pi}{2}$, that is, $A$ is sectorial for all angles $\phi\in (\frac{\pi}{2},\pi)$, but $\psi$ is only bounded on the sector $S_{\frac{\pi}{2}}$. This situation occurs for example when $-A$ is the generator of a bounded $C_0$-semigroup $T$ and $\psi (z) = e^{-z} -1$. The connection between interpolation theory and semigroups was discovered by J.~L.~Lions; see \cite[Section 1.13]{Tr95} for references.  For example, the real interpolation spaces  $(X,\dom{A})_{\theta,q}$ admit the following semigroup characterisation \cite{Tr92}: \emph{if  $\theta\in (0,1)$, $q\in [1,\infty )$, then 
\[
(X, \dom{A})_{\theta,q} = \left\{x\in X: \int_0^\infty t^{-\theta  q} |(T(t)-I)x|_X^q \frac{\ud t}{t}  <\infty \right\}.
\]}
A simple modification of the proof allows us to give the following extension of this result. 

\begin{theorem}\label{semigroup}
 Let $-A$ be the generator of a bounded $C_0$-semigroup $T$ on a Banach space $X$. Then for every Banach function space $\Phi \in \mathfrak{L}_P$,
 \begin{equation}\label{semig.rep}
  (X, \dom{A})_{\Phi} = \left\{x\in X:  \bigl[(0,\infty)\ni t \mapsto t^{-1}|(T(t)-I) x|_X\right] \in \Phi \bigr\}
 \end{equation}
and
\[ 
|x|_{(X, \dom{A})_{\Phi}}\simeq |x|_X + \bigl\| (\cdot)^{-1}|(T(\cdot)-I) x|_X \bigr\|_{\Phi} \quad 
 \bigl(x\in (X, \dom{A}\bigr)_{\Phi}).
\]
\end{theorem} 

The proof follows the lines of \cite[Theorem 1.13.2]{Tr95}. For the convenience of the reader we sketch its pattern. 

\begin{proof}
An analysis of the proof of \cite[Theorem 1.13.2]{Tr95} shows 
that the couple $(X, \dom{A})$ is {\em quasi-linearizable} with respect to the operator-valued functions $V_0$ and $V_1$, where $V_0(t)x:= \frac{1}{t} \int_0^t(T(s)x-x)\,\ud s$ and $V_1(t)x:= x-V_0(t)x$ ($x\in X$, $t>0$). Therefore, there exists a constant $C$ such that
\[
K(t,x)\leq |V_0(t)x|_X + t|V_1(t)x|_{\dom{A}} \leq C K(t,x)\quad \quad (x\in X, \, t>0).
\]
Fix $t>0$ and $x\in X$. Note that $tAV_1(t) x = A\int_0^t T(s) \ud s = T(t)x-x$. Thus, 
\[
t|V_1(t)x|_{\dom{A}}\leq \left| T(t)x -x \right|_X + \sup_{s>0}|T(s)|_{\cL(X)}\, t |x|_X.
\] 
Moreover, note that $|V_0(t)x|_X\leq t P\left((\cdot)^{-1}| T(\cdot)x -x |_X\right)(t)$, and recall that $K(t, x)\leq |x|_X$. Consequently, we get 
\[
K(t,x)\leq C \, \bigl(t P\left((\cdot)^{-1}| T(\cdot)x -x |_X\right)(t) +\min(1,t) |x|_X + | T(t)x -x |_X\bigr),
\] 
where $C$ is a constant independent of $x$ and $t$. Since also $| T(t)x -x |_X\leq C \, K(t,x)$, by applying Lemma \ref{lem.chi}, the proof is complete.
\end{proof}

Our last result in this section extends a theorem of Dore on the boundedness of the $H^\infty$-functional calculus for sectorial operators. Dore proved in \cite{Do99} that every invertible sectorial operator $A$ on a Banach space $X$ has a bounded $H^\infty$-functional calculus in each real interpolation space $(X, \dom{A})_{\theta, q}$ ($\theta \in (0,1)$, $q\in [1,\infty]$). Dore's proof is based on the reiteration theorem. Subsequently, Haase provided an alternative approach which does not rely on the reiteration theorem and gives also an extension of Dore's result to the class of injective sectorial operators and the spaces $(X, \dom{A^\alpha}\cap \rg{A^\alpha})_{\theta, q}$; see \cite[Theorem 6.5.3 and Corollary 6.5.8]{Hs06}. 

In \cite{KaKu10}, Kalton and Kucherenko further extended Haase's result by considering general real interpolation spaces  $(X, \dom{A}\cap \rg{A})_{\Phi}$ and by proving that the part of $A$ in these spaces has absolute functional calculus, a property which is even stronger than a bounded $H^\infty$-functional calculus. In this respect, our last result in this section is weaker. For a comparison of the setting here and in Kalton \& Kucherenko, see Remark \ref{rem.kk}.   

In the case of invertible sectorial operators we extract a simple alternative proof of Dore's theorem, which extends to our general setting.\\

Let $A$ be an injective, sectorial operator, and let $\Phi\in \mathfrak{L}_P$. Put $e(z) := \frac{z}{(1+z)^2}$ ($z\in\C$), so that $fe\in H^\infty_0 (S_\phi )$ for every $f\in H^\infty (S_\phi )$. For every $f\in H^\infty (S_\phi )$ ($\phi \in (\phi_A ,\pi )$) one then defines
\[
 f(A) := A^{-1} (I+A)^2 (fe) (A) 
\]
with maximal domain, so that $f(A)$ is a closed operator. We say that $A$ has a bounded $H^\infty(S_\phi)$-functional calculus if there exists a constant $C\geq 0$ such that 
\[
|f(A)x|_{X}\leq C\, \| f\|_{H^\infty} \, |x|_X \quad (f\in H^\infty(S_\phi),\, x\in \dom{f(A)} ).
\]
Note that, by Theorem \ref{fcrepresent}, since the resolvent of $A$ commutes with $\psi(tA)$, the part of a sectorial operator $A$ in $(X, \dom{A})_{\Phi}$ is again a sectorial operator. 

\begin{theorem}[Dore type theorem] \label{general Dore th}
Let $A$ be an invertible, sectorial operator on a Banach space $X$. Then, for every Banach function space $\Phi\in \mathfrak{L}_{P+Q}$ and every $\phi\in (\phi_A, \pi)$, the part of $A$ in $(X,\dom{A})_{\Phi}$ has a bounded $H^\infty (S_\phi)$-functional calculus.
\end{theorem}

Let $A_{\Phi}$ denote the part of $A$ in $(X, \dom{A})_{\Phi}$.
Furthermore, note that $f(A_\Phi)$ is the part of $f(A)$ in $(X, \dom{A})_\Phi$.

\begin{proof}[{Proof of Theorem \ref{general Dore th}}]
Fix $\alpha\in (0,1)$. Let $\gamma(z):=z^\alpha (1+z)^{-1}$ and $\psi(z):=z(1+z)^{-1}$ ($z\in S_\phi$). Note that $\psi, (\cdot)^{-1}\psi\in \cE(S_\phi)$, and $\gamma\psi, (\cdot)^{-1}\gamma\psi \in H_0^\infty(S_\phi)\subseteq \cE(S_\phi)$.
Let $f\in H^\infty(S_\phi)$.
By algebraic properties of holomorphic functional calculus, see for example \cite[Theorem 1.3.2]{Hs06}, we get
\begin{equation}\label{alg}
 (\gamma\psi)(tA)f(A)x = (\gamma_t f)(A) \psi(tA)x\quad \quad \bigl(t>0,\, x\in \dom{f(A_\Phi)}\bigr),
\end{equation}
where $\gamma_t(z):=\gamma(tz)$ ($z\in S_\phi$). 
Moreover, since $\gamma_tf\in H_0^\infty(S_\phi)$, we have that  
\[
(\gamma_t f)(A)=  \frac{1}{2\pi i}\int_{\partial S_\beta} f(z) \gamma_t(z) R(z,A) \,\ud z\quad \quad (t>0), 
\]
where $\beta\in (\phi_A, \phi)$. 
Thus, if $M:=\sup_{z\in  \partial S_\beta} |z R(z,A) |_{\cL(X)}$, then
\begin{align*}
|(\gamma_t f)(A)|_{\cL(X)} & \leq \frac{M \|f\|_\infty }{2\pi }\int_{\partial S_\beta} 
|\gamma_t(z)| \frac{1}{|z|} |\ud z| \\
& = \frac{2M \|f\|_\infty }{2\pi }\int_{0}^\infty  \frac{t^\alpha r^\alpha}{|1+tre^{i\beta}|} \frac{1}{r} \,\ud r\\
& = \frac{2M \|f\|_\infty }{2\pi }\int_{0}^\infty  \frac{s^\alpha}{|1+se^{i\beta}|} \frac{1}{s}\, \ud s \\
& \leq C_{\alpha,  A} \, \|f\|_\infty .
\end{align*}
This inequality, combined with the estimate (\ref{alg}), yields
\[
|(\gamma\psi)(tA)f(A)x|_X \leq  C_{\alpha, A} \, \|f\|_\infty |\psi(tA)x|_X 
\quad \quad \bigl(x\in \dom{f(A_\Phi)} , \, t>0 \bigr).
\]
Therefore, by Theorem \ref{fcrepresent2},  the proof is complete.
\end{proof}

\begin{remark}\label{CKvsKK} 
The question arises whether $f(A_\Phi)$ is densely defined and thus extends to a bounded operator on $(X,\dom{A})_\Phi$. In this context, note that if $A$ is invertible, then $A_\Phi$ is invertible, too. Furthermore, $\dom{A_\Phi}$ is a core for $f(A_\Phi)$, so that the question above reduces to the question whether $A_\Phi$ is densely defined. By Theorem \ref{fcrepresent2}, this question has a positive answer if $A$ is densely defined and $\Phi$ has absolutely continuous norm.
\end{remark}

\section{Weighted inequalities for the Hardy operator} \label{sec.weights} 

In this section we identify subclasses of $\mathfrak{L}_P$ and $\mathfrak{L}_{P+Q}$ which play a role in applications to Cauchy problems. Throughout this section, let $\E$ be a rearrangement invariant Banach function space over $(\R_+,\ud t)$ equipped with the Banach function norm $\|\cdot\|_{\E}$. Examples of rearrangement invariant Banach function spaces are the $L^p$ spaces ($p\in (1,\infty)$), the Lorentz spaces $L^{p,q}$ ($p$, $q\in [1,\infty )$), and the Orlicz spaces $L^\Phi$.

A measurable function $w:(0,\infty )\rightarrow (0,\infty )$ is called a {\em weight } on $(0,\infty )$. We assume throughout that weights are locally integrable on $(0,\infty)$, and sometimes restrict to weights which are locally integrable on $\R_+ = [0,\infty )$.

Denote by $\cM^+$ the cone of all nonnegative, measurable functions in $\cM$. Given a weight $w$ on $(0,\infty )$, we define for every function $f\in\cM^+$ its {\em distribution function} $w_f : (0,\infty ) \to \R_+$ by
\[
 w_f (\lambda ) := w (\{ |f| >\lambda \} ) \qquad (\lambda >0 ) ,
\]
and its {\em decreasing rearrangement} with respect to the weighted Lebesgue measure $w\,\ud t$, $f^*_w : \R_+ \to \R_+$ by
\[
 f^*_w (t) := \inf \{ \lambda > 0 : w_f (\lambda ) \le t\}  \qquad (t\ge 0 ) .
\]
Then we define $\|\cdot\|_{\E_w}:\cM^+\rightarrow [0,\infty]$ by $\| f\|_{\E_w} := \|f^*_w\|_{\E}$. By \cite[Theorem 4.9, Chapter 2]{BeSh88}, $\|\cdot\|_{\E_w}$ is a rearrangement invariant Banach function norm over $(\R_+, w\,\ud t)$. We put
\[
 \E_w := \{ f\in\cM : \| f\|_{\E_w} <\infty \} ,
\]
and equip this space with the semi-norm $\|\cdot\|_{\E_w}$. By passing to a quotient space we obtain in this way a Banach space, which is by abuse of notation again denoted by $\E_w$. Note that, in the case of $\E = L^p$ ($p\in [1,\infty )$), the weighted space $L^p_w$ thus defined coincides with the usual weighted space $L^p (\R_+ ,w\,\ud t)$. It follows from the definition that the spaces $\E_w$ are always order ideals in $\cM$ and they are thus Banach function spaces in the sense of this article. 

Given a rearrangement invariant Banach function space $\E$, we denote by $M_\E$  (resp. $M^{\E} )$  the class of all weights $w\in L^1_{loc} ((0,\infty) )$ such that $\E_w \subseteq L^1_{loc}$ (resp. $\E_w$ is a subset of the maximal domain of $Q$) and such that the Hardy operator $P$ (resp. its adjoint $Q$) is bounded on $\E_w$. Moreover, we set $C_\E:=M_\E\cap M^\E$ for the class of weights such that the {\em Calder\'on operator} $P+Q$ is bounded on $\E_w$; necessarily we require here $\E_w \subseteq L^1 (\R_+ ; \min (1,t^{-1})\, \ud t)$.

\begin{lemma} \label{lem.chi}
Let $\E$ be a rearrangement invariant Banach function space, and let $w$ be a weight which is locally integrable on $(0,\infty )$. Then $w$ is locally integrable on $\R_+$ if and only if $\chi_{(0,1)} \in\E_w$.
\end{lemma}

\begin{proof}
Note that $\chi_{(0,1)}\in\E_w$ if and only if $(\chi_{(0,1)})^*_w = \chi_{(0,\int_0^1 w)} \in\E$ if and only if $\int_0^1 w$ is finite.  
\end{proof}

By \cite[Theorem 1]{Mu72a}, a weight $w$ on $(0,\infty)$ belongs to $M_{L^p}=:M_p$ ($p\in (1,\infty )$) if and only if it satisfies the condition
\[
 [w]_{M_p}:=\sup_{r>0} \left( \int_r^\infty \frac{w(s)}{s^p}\, \ud s\right) \left( \int_0^r w^{1-p'}(s)\, \ud s \right)^{p-1}<\infty.\tag{$M_p$}
\]
For $p=1$, $w\in M_1$ if and only if 
\[
[w]_{M_1}:=\|Qw/w\|_{L^\infty(\R_+)}<\infty
\tag{$M_1$}
\]
The classes $M^p$ for $p\in[1,\infty)$ are defined analogously and one has $M^p=\{w^{1-p}: w\in M_{p'}\}$ with $[w]_{M^p}=[w^{1/1-p}]_{M_{p'}}$ for $p\in(1,\infty)$, and $w\in M^1$ if and only if $\|Pw/w\|_{L^\infty(\R_+)} <\infty$. Weighted estimates for the Hardy operator $P$ in Lorentz spaces have been studied by several authors; see, for example, Mart{\'{\i}}n \& Milman \cite{MaMi06} and the references therein. To our best knowledge, the description of the classes $M_\E$ and $M^\E$ for an arbitrary rearrangement invariant function space $\E$ is not provided in the literature. The following result, however, identifies a large subset of weights included in $M_\E$ and $M^\E$; see also  Remark \ref{quest} below. 

\begin{theorem}\label{pq operators}
 Let $\E$ be any rearrangement invariant Banach function space over $(\R_+, \ud t)$ with 
 Boyd indices $p_\E$, $q_\E\in (1, \infty)$. Then the following statements hold:
 \begin{itemize}
  \item [(a)] For every weight $w\in\bigcup_{q<p_\E}M_q$ the operator $P$ is bounded on $\E_w$, that is, 
\[
\bigcup_{q<p_\E} M_q \subseteq M_\E .
\]
  \item [(b)] For every weight $w\in \bigcup_{q<p_\E}M^q$ the operator $Q$ is bounded on $\E_w$, that is,
\[
 \bigcup_{q<p_\E}M^q \subseteq M^\E .
\]
  \item [(c)] For every weight $w\in \bigcup_{q<p_\E}M_q\cap M^q$, the Calder\'on operator $P+Q$ is bounded in $\E_w$, that is,
\[
 \bigcup_{q<p_\E}M_q\cap M^q \subseteq C_\E .
\]
 \end{itemize}
\end{theorem}

\begin{proof} 
(a) Fix $q\in (1,p_\E )$ and let $w\in M_q$. Choose $r\in (q_\E ,\infty )$. Then $q<r$ and therefore $M_q \subseteq M_r$. In particular, by \cite[Theorem 4.11, p.~223]{BeSh88} (see also \cite[Theorem 8]{Ca66}),
  $P$ is of joint weak type $(q,q;r,r)$ with respect to $(\R_+, w\,\ud t)$. More precisely, according to \cite[Definition 5.4, p.\,143]{BeSh88}, for every $f\in L^{r,1}_w+L^{q, 1}_w$ and every $t>0$
\[ 
(Pf)^*_w(t) \leq C S_\sigma (f^*_w)(t),
\]
where $S_\sigma$ stands for the corresponding Calder\'on operator associated with the interpolation segment $\sigma:=(q^{-1},q^{-1};r^{-1},r^{-1})$ (see \cite[p.~142]{BeSh88}). By Boyd's theorem \cite[Theorem 5.16, p.\,153]{BeSh88}, $S_\sigma$ is bounded on $\E$. Therefore, we obtain
\[ 
\|Pf\|_{\E_w} =\|(Pf)^*_w\|_{\E}\leq C \|S_\sigma f^*_w\|_{\E}\leq 
 C\|S_\sigma\|_{\cL(\E)}  \| f\|_{\E_w}
\]
for every $f\in L^{r,1}_w+L^{q, 1}_w(\R_+ ) $. Since $\E_w\subseteq L^{r,1}_w +L^{q, 1}_w$ (see, for example, \cite[Lemma 4.2]{ChKr14}), this inequality yields boundedness of $P$ on $\E_w$ and thus $w\in M_\E$.
 
Since $M^q\subseteq M^p$ for every $q<p$ (see, for example, \cite[Proposition 2.9]{BaMiRu01}), the proof of (b) follows exactly the same argument as above. The statement (c) is an immediate consequence of (a) and (b) and the definition of $C_\E$.
\end{proof}

As in \cite{ChKr15}, we also introduce the classes $A^-_p = A^-_p (\R_+)$ and $A^+_p = A^+_p (\R_+ )$ of Muckenhoupt-Sawyer type weights associated with one-sided maximal functions on the half-line. First, given $p\in (1,\infty )$, we say that a weight $w$ belongs to $A^-_p$, if
\[
[w]_{A_p^-}:=\sup_{0\leq a<b<c }\frac{1}{(c-a)^p}
\left( \int_b^c w \, \ud t \right) \left( \int_a^bw^{1-p'} \ud t \right)^{p-1}<\infty. \tag{$A_p^-$}
\]
Second, we set $A^+_{p}:=\{w^{1-p}: w\in A^-_{p'}\}$ ($p\in (1, \infty)$). Note that every decreasing function $w:(0,\infty)\rightarrow (0,\infty)$ belongs to $A_p^-$ for every $p\in(1,\infty)$. Moreover, $A_p^-\subseteq M_p$ for every $p\in(1,\infty)$; in order to see this, use also that for every $p\in (1,\infty )$ and every weight $w\in A^-_p$ one has $\E_w \subseteq L^r_{loc} (\R_+ )$ for some $r>1$ \cite[Proposition 4.2(ii)]{ChKr15}. Since, in addition, the classes $A_p^-$ possess the so-called openness property, that is, $\bigcup_{q<p}A_q^-=A_p^-$ (see \cite[Lemma 2.5]{ChKr15}), we get the following consequence of Theorem \ref{pq operators}.

\begin{corollary}\label{A-p}
Let $\E$ be a rearrangement invariant Banach function space over $(\R_+, \ud t)$ with Boyd indices $p_\E$, $q_\E\in (1, \infty)$. Then, $A^-_{p_\E}\subseteq M_\E$ and $A^+_{p_\E}\subseteq M^\E$. In particular, $A^-_{p_\E}\cap A^+_{p_\E}\subseteq C_\E$.
\end{corollary}

At the end of the section, we make some complementary remarks to Theorem \ref{pq operators} and Corollary \ref{A-p}, that is, we collect some more relations between the classes $M_p$, $C_p$, and $A^-_p$, $A^+_p$. 

\begin{proposition}\label{connection} 
The following assertions hold:
\begin{itemize}
 \item [(a)] The classes $M_{p}$ and $M^{p}$ ($p\in (1,\infty)$) do not have the openness property in the sense that $\bigcup_{q<p} M_q \subsetneqq M_p$ and $\bigcup_{q<p} M^q \subsetneqq M^p$ for every $p\in(1,\infty)$. In particular, the inclusions $A_p^- \subseteq M_{p}$ and $A^+_p\subseteq M^p$ for $p\in (1, \infty)$ are proper.
 \item [(b)] For every $p\in (1,\infty)$, the inclusions $A_{p}^-\cap A_{p}^+ \subseteq C_p \subseteq A_p^-\cap L_{loc}^1 \subseteq A_p^-$ are proper.
\end{itemize}
\end{proposition}

\begin{proof}
 (a) Note first that the openness property for the classes $M_p$ is equivalent to the openness property for the classes $M^p$.   However, the lack of the openness property for the class $C_p=M_p\cap M^p$ has recently been shown in in Duoandikoetxea, Martin-Reyes \& Ombrosi \cite[Proposition 4.1]{DuMROm13}. Since $M_q\subseteq M_p$ and $M^q\subseteq M^p$ for all $p$, $q\in (1,\infty )$ with $q<p$, and since the classes $\bigcup_{q<p} M^q$ $(p\in(1,\infty))$ do possess the openness property, the two inclusions in the statement are strict. We remark in addition to this abstract argument that a straightforward calculation shows that the weight constructed in the proof of \cite[Proposition 4.1]{DuMROm13} illustrates this fact, too. 
  
 (b) A similar argument as in (a) shows that the first inclusion is proper. For the last one, recall that every decreasing function on $(0,\infty)$ is in $A_{p}^-$, but not necessarily locally integrable on $\R_+$. In order to show that $C_p \subsetneqq A_p^-\cap L_{loc}^1$, one may check that for every $\alpha \in(-1,0]$ the weight 
\[
w(t):=  \left\{\begin{array}{cc} t^\alpha, & 0<t<1,\\
 t^{-1},& t\geq 1,
 \end{array}\right. 
\]
belongs to $\bigcap_{1<p<\infty} A^-_p$ but not to any $C_p$ ($p\in (1,\infty)$).
\end{proof}

 \begin{remark}\label{quest}
The question whether the inclusions $M_{p_\E}\subseteq M_\E$ and / or $C_{p_\E}\subseteq C_\E$ hold for any rearrangement invariant Banach function space $\E$ over $(\R_+, \ud t)$ with Boyd indices $p_\E$, $q_\E\in (1,\infty)$ is left open.
\end{remark}

\section{The maximal regularity} \label{sec.appl}

Let $A$ be a closed linear operator on a Banach space $X$. Let $\E$ be a rearrangement invariant Banach function space over $(\R_+ ,\ud t)$, and let $w$ be a weight on $(0,\infty )$. We denote by $\E_{w,loc} (X)$ the space of all (equivalence classes) of measurable functions $f :\R_+ \to X$ such that $|f|_X\, \chi_{(0,\tau )} \in\E_w$ for every $\tau >0$. We say that the first order Cauchy problem
\begin{equation} \label{first order}
\dot{u} + A u = f \text{ on } \R_+ , \quad u(0) = 0,
\end{equation} 
has {\em $\E_w$-maximal regularity} if for each right-hand side $f\in \E_{w, loc}(X)$ there exists a unique function $u\in  W^{1,1}_{loc}(X)$ such that $\dot u$, $Au\in \E_{w, loc}(X)$, and such that $u$ solves (\ref{first order}).

The theory of maximal regularity of abstract linear evolution equations plays an important role in applications to nonlinear problems. Its study combines results and techniques from harmonic analysis, Fourier analysis, theory of singular integral operators, operator theory, geometry of Banach spaces and interpolation theory. Whereas in the definition of $\E_w$-maximal regularity, and thus in the Cauchy problem \eqref{first order} above, one considers only zero initial values, the following result treats the full Cauchy problem with nonzero initial values and right-hand sides. 

\begin{theorem} \label{m.r.extrap}
Assume that the problem \eqref{first order} has $L^p$-maximal regularity for some $p\in (1,\infty )$. Then, for every rearrangement invariant Banach function space $\E$ over $(\R_+, \ud t)$ with Boyd indices $p_\E$, $q_\E\in(1,\infty)$, for every weight $w \in A_{p_\E}^-(\R_+)$, for every $f\in \E_{w,loc} (X)$ and every $x\in (X,\dom{A})_{\E_w}$ the problem
\begin{equation} \label{first order.initial}
\dot{u} + A u = f \text{ on } \R_+ , \quad u(0) = x ,
\end{equation} 
admits a unique solution $u\in W^{1,1}_{loc} (X)$ satisfying $\dot u$, $Au\in \E_{w,loc} (X)$.
\end{theorem}

We point out that for the initial value $x=0$ the above theorem reduces to the extrapolation result \cite[Theorem 5.1]{ChKr15} which says that $L^p$-maximal regularity for some $p\in (1,\infty )$ implies $\E_w$-maximal regularity for every rearrangement invariant Banach function space $\E$ and every weight $w$ as in the statement. Note also that, by definition of the trace space $[X,\dom{A}]_{\E_w}$ (and since $\dom{A}\subseteq X$), the condition on the initial value $x$ is necessary for the conclusion (use that the class $A_{p_\E}^-$ is contained in $M_\E$ by Corollary \ref{A-p} and that the interpolation spaces $[X,\dom{A}]_{\E_w}$ and $(X,\dom{A})_{\E_w}$ coincide by Theorem \ref{trace=interp}). The point of Theorem \ref{m.r.extrap} is that the condition $x\in (X,\dom{A})_{\E_w}$ is actually sufficient for the conclusion.

\begin{proof}[Proof of Theorem \ref{m.r.extrap}]
{\em Existence.} Assume that $A$ has $L^p$-maximal regularity for some $p\in (1,\infty )$, and let $\E$, $w$, $f$ and $x$ be as in the statement. Then there exists $v\in W^{1,\E_w} (X,\dom{A})$ such that $v(0) = x$. By \cite[Theorem 5.1]{ChKr15}, there exists $z\in W^{1,1}_{loc} (X)$ such that $\dot z$, $Az\in \E_{w,loc} (X)$ and  
\[
 \dot{z} + Az = -\dot{v} -Av +f \text{ on } \R_+, \quad z(0) = 0 .
\]
Now $u:=v+z$ is a desired solution of \eqref{first order.initial}. 

{\em Uniqueness} follows from the unique solvability of \eqref{first order} with $f=0$, that is, from the assumption of $L^p$-maximal regularity and \cite[Theorem 5.1]{ChKr15}.
\end{proof}

\providecommand{\bysame}{\leavevmode\hbox to3em{\hrulefill}\thinspace}

\bibliographystyle{amsplain}

\begin{thebibliography}{10}

\bibitem{ArMu90}
M.~A. Ari{\~n}o, B. Muckenhoupt, \emph{Maximal functions on
  classical {L}orentz spaces and {H}ardy's inequality with weights for
  nonincreasing functions}, Trans. Amer. Math. Soc. \textbf{320} (1990),
  727--735.

\bibitem{BaMiRu01}
J. Bastero, M. Milman, F.~J. Ruiz, \emph{On the connection
  between weighted norm inequalities, commutators and real interpolation}, Mem.
  Amer. Math. Soc. \textbf{154} (2001). 

\bibitem{Be74II}
C. Bennett, \emph{Banach function spaces and interpolation methods. {II}.
  {I}nterpolation of weak-type operators}, Linear operators and approximation,
  {II} ({P}roc. {C}onf., {M}ath. {R}es. {I}nst., {O}berwolfach, 1974),
  Birkh\"auser, Basel, 1974, pp.~129--139. Internat. Ser. Numer. Math., Vol.
  25.

\bibitem{BeSh88}
C. Bennett and R. Sharpley, \emph{Interpolation of operators}, Pure and
  Applied Mathematics, vol. 129, Academic Press Inc., Boston, 1988.
 

\bibitem{BrKr91}
Yu.~A. Brudny{\u\i} and N.~Ya. Krugljak, \emph{Interpolation functors and
  interpolation spaces. {V}ol. {I}}, North-Holland Mathematical Library,
  vol.~47, North-Holland Publishing Co., Amsterdam, 1991, Translated from the
  Russian by Natalie Wadhwa.

\bibitem{BuBe67}
P.~L. Butzer and H.~Berens, \emph{Semi-groups of operators and approximation},
  Springer Verlag, Berlin, Heidelberg, New York, 1967.

\bibitem{Ca66}
A.-P. Calder{\'o}n, \emph{Spaces between {$L^{1}$} and {$L^{\infty }$} and the
  theorem of {M}arcinkiewicz}, Studia Math. \textbf{26} (1966), 273--299.


\bibitem{ChFi14}
R.~Chill, A.~Fiorenza, \emph{Singular integral operators with
  operator-valued kernels, and extrapolation of maximal regularity into
  rearrangement invariant {B}anach function spaces}, J. Evol. Eq. \textbf{12} (2014), 795-828.

\bibitem{ChKr15}
R. Chill, S. Kr\'ol, \emph{Weighted inequalities for singular
  integral operators on the half-line}, Preprint (2014).

\bibitem{ChKr14}
R. Chill, S. Kr\'ol,  \emph{Extrapolation of {$L^p$}-maximal regularity for second order
  {C}auchy problems}, Perspectives in operator theory, Banach Center Publ.,
  vol.~xx, Polish Acad. Sci., Warsaw, 2016, p.~xxx.

\bibitem{CrMaPe11}
D.~V. Cruz-Uribe, J.~M. Martell, and C. P{\'e}rez,
  \emph{Weights, extrapolation and the theory of {R}ubio de {F}rancia},
  Operator Theory: Advances and Applications, vol. 215, Birkh\"auser/Springer
  Basel AG, Basel, 2011. 
  
\bibitem{CGMP06}
G.~P. Curbera, J. Garc{\'{\i}}a-Cuerva, Jo.~M.
  Martell,  C. P{\'e}rez, \emph{Extrapolation with weights,
  rearrangement-invariant function spaces, modular inequalities and
  applications to singular integrals}, Adv. Math. \textbf{203} (2006),
  256--318.

\bibitem{Do99}
G. Dore, \emph{{$H^\infty$} functional calculus in real interpolation
  spaces}, Studia Math. \textbf{137} (1999), 161--167. 

\bibitem{DuMROm13}
J. Duoandikoetxea, F.~J. Mart{\'{\i}}n-Reyes, S. Ombrosi,
  \emph{Calder\'on weights as {M}uckenhoupt weights}, Indiana Univ. Math. J.
  \textbf{62} (2013), 891--910.

\bibitem{HaHaKu06}
B.~H. Haak, M. Haase, P.~C. Kunstmann, \emph{Perturbation,
  interpolation, and maximal regularity}, Adv. Differential Equations
  \textbf{11} (2006), 201--240. 

\bibitem{Hs06}
M. Haase, \emph{The functional calculus for sectorial operators}, Operator
  Theory: Advances and Applications, vol. 169, Birkh\"auser Verlag, Basel,
  2006. 

\bibitem{KaKu07}
N.~J. Kalton and T.~Kucherenko, \emph{Sectorial operators and interpolation
  theory}, Interpolation theory and applications, Contemp. Math., vol. 445,
  Amer. Math. Soc., Providence, RI, 2007, pp.~111--119.

\bibitem{KaKu10}
N.~J. Kalton and T.~Kucherenko, \emph{Operators with an absolute functional calculus}, Math. Ann.
  \textbf{346} (2010), no.~2, 259--306.

\bibitem{Ka75}
T.~F. Kalugina, \emph{Interpolation of {B}anach spaces with a functional
  parameter. {R}eiteration theorem}, Vestnik Moskov. Univ. Ser. I Mat. Meh.
  \textbf{30} (1975), 68--77. 
  
\bibitem{Ko67}
H. Komatsu, \emph{Fractional powers of operators. {II}. {I}nterpolation
  spaces}, Pacific J. Math. \textbf{21} (1967), 89--111. 

\bibitem{Ko72}
H. Komatsu, \emph{Fractional powers of operators. {VI}. {I}nterpolation of
  non-negative operators and imbedding theorems}, J. Fac. Sci. Univ. Tokyo
  Sect. IA Math. \textbf{19} (1972), 1--63. 

\bibitem{KrWe14}
C. Kriegler, L. Weis, \emph{Paley-{L}ittlewood decomposition for
  sectorial operators and interpolation spaces},  \textbf{195} (2010),
  231--244. 

\bibitem{KuUl14}
P. Kunstmann, A. Ullmann, \emph{{$\mathcal{R}_s$}-sectorial
  operators and generalized {T}riebel-{L}izorkin spaces}, J. Fourier Anal.
  Appl. \textbf{20} (2014), 135--185. 

\bibitem{LiTz79}
J.~Lindenstrauss and L.~Tzafriri, \emph{Classical {B}anach {S}paces {I, II}},
  Springer Verlag, Berlin, Heidelberg, New York, 1979.

\bibitem{Lu95}
A.~Lunardi, \emph{Analytic {S}emigroups and {O}ptimal {R}egularity in
  {P}arabolic {P}roblems}, Progress in Nonlinear Differential Equations and
  Their Applications, vol.~16, Birkh\"auser, Basel, 1995.

\bibitem{Lu09}
A. Lunardi, \emph{Interpolation theory}, second ed., Appunti. Scuola
  Normale Superiore di Pisa (Nuova Serie). [Lecture Notes. Scuola Normale
  Superiore di Pisa (New Series)], Edizioni della Normale, Pisa, 2009.
  

\bibitem{MaMi06}
J. Mart{\'{\i}}n, M. Milman, \emph{Extrapolation methods and {R}ubio
  de {F}rancia's extrapolation theorem}, Adv. Math. \textbf{201} (2006),
  209--262. 

\bibitem{MN91}
P.~Meyer-Nieberg, \emph{Banach {L}attices}, Springer Verlag, Berlin,
  Heidelberg, New York, 1991.

\bibitem{MeSc12}
M. Meyries, R. Schnaubelt, \emph{Interpolation, embeddings and
  traces of anisotropic fractional {S}obolev spaces with temporal weights}, J.
  Funct. Anal. \textbf{262} (2012), 1200--1229. 

\bibitem{MeSc12a}
M. Meyries, R. Schnaubelt, \emph{Maximal regularity with temporal weights for parabolic problems
  with inhomogeneous boundary conditions}, Math. Nachr. \textbf{285} (2012), 1032--1051. 

\bibitem{MeVe14}
M. Meyries, M.~C. Veraar, \emph{Traces and embeddings of anisotropic
  function spaces}, Math. Ann. \textbf{360} (2014),  571--606.
 

\bibitem{Mu72a}
B. Muckenhoupt, \emph{Hardy's inequality with weights}, Studia Math.
  \textbf{44} (1972), 31--38. 
  
\bibitem{Pe64}
G. Peetre, \emph{Espaces d'interpolation, g\'en\'eralisations,
  applications}, Rend. Sem. Mat. Fis. Milano \textbf{34} (1964), 133--164.
 

\bibitem{Pe68}
J.~Peetre, \emph{A theory of interpolation of normed spaces}, Notas de
  Matem\'atica, No. 39, Instituto de Matem\'atica Pura e Aplicada, Conselho
  Nacional de Pesquisas, Rio de Janeiro, 1968. 

\bibitem{Sa81}
Y.~Sagher, \emph{Real interpolation with weights}, Indiana Univ. Math. J.
  \textbf{30} (1981), no.~1, 113--121. 

\bibitem{Tr83}
H.~Triebel, \emph{Theory of {F}unction {S}paces}, Birkh\"auser, Basel, 1983.

\bibitem{Tr92}
H. Triebel, \emph{Theory of function spaces. {II}}, Monographs in
  Mathematics, vol.~84, Birkh\"auser Verlag, Basel, 1992. 

\bibitem{Tr95}
H. Triebel, \emph{Interpolation theory, function spaces, differential operators},
  second ed., Johann Ambrosius Barth, Heidelberg, 1995. 
  
  
\bibitem{Ya90}
A.~Yagi, \emph{Parabolic equations in which the coefficients are generators of
  infinitely differentiable semigroups {II}}, Funkcial. Ekvac. \textbf{33}
  (1990), 139--150.

\end{thebibliography}

\end{document}